\documentclass[12
pt,twoside]{amsart}
\usepackage{amssymb}
\usepackage{a4wide}
\usepackage[latin1]{inputenc}
\usepackage[T1]{fontenc}
\usepackage{times}

\def\hpq0{h^{p,q}_{\leq 0}}
\def\Hpq0{\H_{\leq 0}^{p,q}}

\def\dbar{\bar\partial}
\def\ddbar{\partial\dbar}

\def\R{{\mathbb R}}

\def\C{{\mathbb C}}

\def\Z{{\mathbb Z}}

\def\V{{\mathcal V}}

\def\H{{\mathcal H}}
\def\E{{\mathcal E}}

\def\L{{\mathcal L}}

\def\Im{{\rm Im\,  }}
\def\L{{\mathcal L}}

\def\be{\begin{equation}}
\def\ee{\end{equation}}

\newtheorem{thm}{Theorem}[section]
\newtheorem{lma}[thm]{Lemma}
\newtheorem{cor}[thm]{Corollary}
\newtheorem{prop}[thm]{Proposition}

\theoremstyle{definition}

\theoremstyle{remark}

\newtheorem{preremark}{Remark}
\newtheorem{preex}{Example}

\numberwithin{equation}{section}

\begin{document}

\begin{abstract} 

We give some remarks on  geodesics in the space of K\"ahler metrics that are defined for all time. Such curves are conjecturally induced by holomorphic vector fields, and we show that this is indeed so for regular geodesics, whereas the question for generalized geodesics is still open (as far as we know). We also give a result about the derivative of such geodesics which implies a variant of a theorem of Atiyah and Guillemin-Sternberg on convexity of the image of certain moment maps.

\end{abstract}

\title[]
{Long geodesics in the space of K\"ahler metrics.
 }

\author[]{ Bo Berndtsson}

\bigskip

\maketitle

\quad \quad \quad \quad \quad \quad {\it Dedicated to L\'aszl\'o Lempert, friend and collaborator.}

\section{Introduction}

Let $(X,\omega)$ be a compact K\"ahler manifold of complex dimension $n$. We denote by
$$
\H_\omega=\{u\in C^\infty; i\ddbar u+\omega>0\}
$$
 the Mabuchi space of potentials of K\"ahler metrics in the same cohomology class as $\omega$, \cite{Mabuchi}. A {\it complex geodesic} in this space is a complex curve $\tau\in U\to u_\tau\in \H_\omega$, where $U$ is a domain in $\C$, such that $i\ddbar_{\tau x} u_\tau +\omega \geq 0$ and 
$$
(i\ddbar_{\tau x}u_\tau +\omega)^{n+1}=0.
$$
Here the subscript of the $\ddbar$ operator indicates that it is to be taken with respect to the variables $\tau, x$ jointly. If $u_\tau$ does not depend on the imaginary part of $\tau$, this means that $u_t=u_\tau$ is a geodesic for the Riemannian structure on the Mabuchi space, see \cite{Semmes}, \cite{Donaldson}

For a complex geodesic
$$
i\ddbar_{\tau x}u_\tau 
$$
is a $(1,1)$-form on $U\times X$, where $U$ is an open set in $\C$, which we usually take to be a strip, if  $u_\tau$ is independent of $s$. We also note  that $i\ddbar_{\tau x}u_\tau +\omega\geq 0$ on $U\times X$  implies  that $u_\tau$ is subharmonic in $\tau$ for $x$ fixed, so $u_t$ is a convex function of $t$ if $u_\tau$ is independent of $s$.

A priori, $u_\tau$ should be smooth and strictly $\omega$-plurisubharmonic for $\tau$ fixed,  but as is customary we will allow also less regular functions (albeit always bounded), satisfying $i\ddbar_x u_\tau +\omega \geq 0$. Such functions are called generalized geodesics, and sometimes we will  refer to the bona fide geodesics as 'regular' geodesics.

In this note we will be mainly interested in 'long' geodesics, i. e. geodesics  $u_t$  defined for all $-\infty <t<\infty$. Assuming also that $u_t$ is of class $C^1$ with respect to $t$, we denote by  $\dot u_t$  the derivative of $u_t$ with respect to $t$ and put 
$$
A_t:= \{\dot u_t(x); x\in X\}
$$
and
$$
B_x:=\{\dot u_t(x); -\infty <t<\infty\}.
$$
Thus $A_t$ is the range of the velocity vector for $t$ fixed $\dot u_t(x)$ as $x$ varies, whereas $B_x$ is the range when $t$ varies and $x$ is fixed.
Our first result is as follows:
\begin{thm}
Let $u_t$ be a geodesic of class $C^1$ defined for all $t$ in $\R$ that is strictly $\omega$-plurisubharmonic for each $t$. Then $ A_t=\overline B_x$ for all $t$ and all $x$ outside a pluripolar set in $X$.
\end{thm}

Our assumptions on the geodesic are very strong, but there is at least one class of examples: Let $V$ be a holomorphic vector field with flow $F_\tau(x)$. This is defined for all $\tau \in \C$ (since $X$ is compact), and we can define
$$
\omega_\tau=F_{\tau}^*(\omega).
$$
This a complex curve of K\"ahler forms on $X$, defined for all $\tau$ and $\omega_\tau$ always lies in the same cohomology class as $\omega$ in $H^{1,1}(X)$ since $F_\tau$ is homotopic to $F_0$; the identity map. By the $\ddbar$-lemma for K\"ahler manifolds we can write
$$
\omega_\tau=i\ddbar u_\tau+\omega
$$
for some functions $u_\tau$, uniquely determined up to a function $f(\tau)$. Thus we get a curve in $\H_\omega$, but the function $f(\tau)$ cannot always be chosen so that $u_\tau$ is a (complex)  geodesic.  This is, however, possible if $V$ satisfies the cohomological condition that the contraction of the K\"ahler form with $V$ is $\dbar$-exact,
$$
V\rfloor \omega =\dbar h
$$
for some function $h$. (It is always $\dbar$-closed since $V$ is holomorphic.) This condition is in particular fulfilled if the imaginary part of $V$ is Hamiltonian for the symplectic form $\omega$, and this in turn guarantees also that $u_\tau$ can be chosen to depend only on the real part of $\tau$ - since $\omega_\tau$ then is independent of the imaginary part of $\tau$. Hence all holomorphic vector fields whose imaginary parts are Hamiltonian furnish examples to which Theorem 1.1 applies.

The theorem has some features of an ergodic theorem, saying that averages over space equal averages over time; only here we don't really have averages, but images of maps. Moreover, the theorem says something about the real part of $V$, whereas it is the imaginary part that is measure preserving. A natural question in this context is whether there is an abstract counterpart of this result for measure-preserving maps that can in some sense be complexified. 

We next move to an even more special situation -- geodesics with multidimensional time. By this we mean functions $u_t$, defined for $t$ in an open subset of $\R^k$, and whose restriction to any real line is a (one-variable) geodesic. More precisely, we have functions 
$u_\tau$, with $\tau$ in a tube domain $U$ in $\C^k$ that are $\omega$-plurisubharmonic on $U\times X$ and depend only on $t$, and the restriction of $u_t$ to any real line is a geodesic.  We can then define subsets $A_t$ and $B_x$ of $\R^k$ as before, replacing $\dot u_t$ by the gradient of $u_t$ with respect to $t$. By essentially the same argument as in the one dimensional case and the Hahn-Banach theorem we then show that the same statement holds in the multidimensional case as well (Theorem 2.2). 
Since the gradient image of a convex function on $\R^k$ is always convex, this gives us the corollary that the sets $A_t=:A$ all coincide, and that $A$ is convex.

Just as in the one-dimensional case the main examples come from holomorphic vector fields whose imaginary parts are Hamiltonian, we can apply the higher dimensional case to commuting $k$-tuples of vector fields whose imaginary parts are Hamiltonian. As is well known, the gradient of $u_t$ with respect to $t$ (for, say, $t=0$) is then the moment map of the Hamiltonian action, and we get that the image of the moment map is convex. This is strongly related to the complex case of the celebrated theorem of Atiyah and Guillemin-Sternberg, see \cite{Atiyah}, \cite{Guillemin-Sternberg}.

In the next section we will give the proof of Theorem 1.1 and its multidimensional analog. In the section after that we will discuss when a holomorphic vector field induces a geodesic in the space of K\"ahler metrics, and give proofs of our claims above. Most likely, this material is well known, but I have not been able to find a reference.  In a final section we will discuss if there are other examples of geodesics that are defined for all time $t\in\R$ than the ones induced by holomorphic vector fields, and we will prove that in case the geodesic is regular there are not. This is based on an analysis of the Mabuchi K-energy along such geodesics which may have an independent interest. Briefly, what we will show is the following.

It is well known that the K-energy is convex along geodesics, and that its second derivative can vanish identically only if the geodesic is induced by a holomorphic vector field. For regular geodesics, we will prove that, unless the second derivative vanishes indentically, the metric on $U$ (the strip where the geodesic is defined) defined by
$$
\frac{d^2}{dt^2} K(u_t) d\tau\otimes d\bar\tau
$$
has strictly negative curvature, bounded from above by a negative constant depending only on the volume of $X$ and the dimension. This is clearly impossible if $u_t$ is defined for all time, since we would then have a metric of strictly negative curvature on $\C$.   Therefore the  second derivative vanishes identically, which implies that the geodesic is induced by a holomorphic vector field.

This last part of the paper depends heavily on a result of Dan Burns on Monge-Amp\`ere foliations. It also follows from  recent work of Wan and Wang, \cite{Wan-Wang}, who, among many other  things, give a direct proof of a curvature estimate from which our Theorem 4.2 follows. It seems to be an interesting problem if Theorem 4.2 also holds  for generalized geodesics (in which case it is known that the K-energy is still convex).

Finally I'd like to thank Robert Berman, S\'ebastien Boucksom and Tam\'as Darvas for helpful comments, and the referee for spotting several errors and points that needed to be clarified. 

\section{Theorem 1.1 and its multidimensional analog.}

An important part in our proof is played by a construction from \cite{B1} of a certain measure on $\R$ defined by a a geodesic of class $C^1$. The main observation is that if $u_t$ is such a geodesic defined for $a<t<b$ (i. e.  not necessarily defined for all times), and $f$ is a continuous function on $\R$, then the integrals
$$
\int_X f(\dot u_t) \omega_t^n/n! =:I(f)
$$
do not depend on $t$. Thus they define a measure on $\R$, $\mu$, the push forward of $\omega_t^n/n!$ under the map $\dot u_t: X\to \R$, and the main point is that this is the same for all $t$. In case the geodesic is induced by a holomorphic vector field $V$  with Hamiltonian imaginary part, it is well known (see also the next section), that $\dot u_0$ is the Hamiltonian 
of $\Im(V)$, and in that case $\mu$ is the measure studied by Duistermaat and Heckman, \cite{Duistermaat-Heckman}, in the case when the flow of $\Im(V)$ is periodic. 

If we assume that $\omega_t>0$ for all $t$, this implies that the supremum norm on $\dot u_t$ does not depend on $t$, since it is the right endpoint of the support of $\mu$. It has been proved by Darvas, \cite{Darvas}, that this actually holds even without the positivity assumption on $\omega_t$, and in a suitable formulation even without the $C^1$-assumption on the geodesic.

In our case, when $u_t$ is defined for all $t$ an immediate consequence of this is that 
$$
u_t\leq u_0+C|t|,
$$
for some constant, so we have at most linear growth. We can therefore define a function on $X$
$$
g(x) =\lim_{t\to \infty} u_t(x)/t;
$$
the convexity in $t$ implies that the limit exists. Being the limit of an essentially increasing sequence of functions that are $(1/t)\omega$-plurisubharmonic, the upper semicontinuous regularization of $g$, $g^*$,  is plurisubharmonic on $X$. Since $X$ is compact, $g^*$ is therefore constant, and it follows that $g$ itself is constant outside a pluripolar set. 

Now we recall the definitions from the introduction
$$
A_t:= \{\dot u_t(x); x\in X\}
$$
and
$$
B_x:=\{\dot u_t(x); -\infty <t<\infty\}.
$$
Since $X$ is connected they are both intervals, and we want to prove that for most $x$ they coincide,  after taking the closure of $B_x$. So, fix an $x_0$. Since the push-forward measure $\mu$ is independent of $t$ and $A_t$ is its support, $A_t=:A$ does not depend on $t$. Moreover, $\dot u_t(x_0)\in A_t=A$ for all $t$, so it follows that $B_{x_0}\subset A$.

For the reverse inclusion, let $x_0$ be such that  $g(x_0)=g^*(x_0)$. We use that for any $x$, $g(x)=\lim_{t\to\infty}\dot u_t(x)$, and $\dot u_t$ is increasing in $t$. Since $g\leq g^*=:C$ everywhere it follows that
$$
\dot u_t(x)\leq C=\lim_{t\to\infty }\dot u_t(x_0),
$$
and $\lim_{t\to\infty }\dot u_t(x_0)$ is the right endpoint of the interval $B_{x_0}$. Replacing the geodesic $u_t$ by $u_{-t}$, we find that $A\subset \overline{B_{x_0}}$, possibly after assuming that $x_0$ does not belong to another pluripolar set. This completes the proof of Theorem 1.1.

We next turn to the multidimensional case and the first step is a multidimensional variant of the 'Duistermaat-Heckman'-measure $\mu$ that we used above. We will consider functions $u_t(x)$ defined for $x$ in $X$ and for $t$ in a convex domain in $\R^k$. As before we sometimes extend them to $\tau=t+is \in \C^k$ by putting $u_\tau:=u_t$.  We say that $u_t$ is a multidimensional geodesic in $\H_\omega$ if $u_\tau$ is $\omega$-plurisubharmonic as a function of $(\tau,x)$ and  the restriction of $u_t$ to any line in $\R^k$ is a geodesic in $\H_\omega$.

\begin{thm} Let $u_t$ be a multidimensional geodesic in $\H_\omega$, of class $C^2$ in $t$. Let $f$ be a continuous function on $\R^k$, and denote by $\partial_tu_t$ the gradient of $u_t$ with respect to $t$.
Then the integrals
$$
\int_X f(\partial_t u_t) \omega_t^n/n!
$$
do not depend on $t$. In other words, the push-forward measures of $\omega_t^n/n!$ under the maps $\partial_t u_t$ do not depend on $t$ and give a well defined measure $d\mu$ on $\R^k$. 
\end{thm}
The proof is essentially the same as in one variable and consists simply in differentiating with respect to $t$ - we may of course assume that $f$ is of class $C^1$ in the proof. Recall that the geodesic equation in one variable can be written
$$
\ddot u_{t t} -|\dbar \dot u_t|_t^2 =0,
$$
where $\ddot u_{t t}$ is the second derivative and $|\cdot|_t $ means the (pointwise) norm with respect to the K\"ahler metric $\omega_t$. This means that if in the multidimensional case $u_t$ is geodesic along each line, then
$$
\ddot u_{i j} -\langle \dbar \dot u_i, \dbar \dot u_j\rangle_t=0
$$
for all $(i j)$, where we omit the $t$ in the subscripts to make room for the variables with respect to which we differentiate. Differentiating the integral with respect to $t_i$ we get
$$
\int_X \sum_j f'_j\ddot u_{i j} \omega_t^n/n! +\int_X fi\ddbar\dot u_i\wedge \omega_t^{n-1}/(n-1)!.
$$
Applying Stokes' theorem to the second integral we see that this vanishes, which completes the proof of Theorem 2.1.

Now we define the sets $A_t$ and $B_x$ as before, replacing  $\dot u_t$ by the gradient with respect to $t$.
\be
A_t:= \{\partial_t u_t(x); x\in X\}
\ee
and
\be
B_x:=\{\partial_t u_t(x); -\infty <t<\infty\}.
\ee

 Exactly as in one variable it follows from Theorem 2.1 that the sets $A_t=:A$ are all identical since they are equal to the support of $d\mu$, and that the closure of  $B_x$ is always a subset of $A$. Moreover, the closure of $B_x$ is always convex, since $B_x$ is the gradient image of a convex function. This is a standard fact that can be found e.g. in \cite{Gromov}. What remains is to prove that the closure of $B_x$ fills out all of $A$. 

For this we consider points $\lambda$ on the unit sphere in $\R^k$ and the geodesics
$$
s\to \phi_{s\lambda}=: \Phi^{\lambda}_s.
$$
The derivative  is
$$
\dot\Phi^\lambda=\lambda\cdot\partial_t\phi|_{s\lambda}.
$$
Then 
$$
B^\lambda_x=\{\dot\Phi^\lambda_s(x), s\in\R\}=\{\lambda\cdot\xi;\xi\in B_x\},
$$
and
$$
A^\lambda_s=\{\dot\Phi^\lambda_s(x), x\in X\}.
$$
By the result in one variable we have $A^\lambda_t=A^\lambda=B^\lambda_x$ for all $x$ outside some pluripolar set, depending on $\lambda$. This means precisely that the orthogonal projection of $A$ onto the line $\{s\lambda\}$ equals the projection of $B_x$. If we restrict attention to a countable dense set of $\lambda$'s this holds for $x$ outside of a single pluripolar set in $X$.  This is enough to show that $A$ is included in the closure of  $B_x$ for all $x$ outside this set. Indeed,  if there were a point of $A$ outside $\overline{B_x}$ , we could find a hyperplane separating  $\overline{B_x}$ from that point (here we are using that $\overline{B_x}$ is convex). 
We can also assume this hyperplane is defined by a point $\lambda$ in the countable dense set, and then get a contradiction to the one-dimensional result.

Summing up we have proved

\begin{thm} Let $u_t$ be a multidimensional geodesic in $\H_\omega$, defined for $t$ in all of $\R^k$. Then the sets $A_t$ defined in (2.1) are all equal; $A_t=:A$, and they coincide with the closure of $B_x$ for all $x$ outside a pluripolar set. As a consequence, the set $A$ is convex.
\end{thm} 

\section{ Holomorphic vector fields that induce geodesics.}

Let $V$ be a (nontrivial) holomorphic vector field on $X$, and let $F_\tau$ be its flow for (complex) time $\tau$. Put $\omega_\tau = F_\tau^*(\omega)$. The question we discuss in this section is when there is a complex geodesic $u_\tau$ in $\H_\omega$ such that
$$
\omega_\tau=i\ddbar u_\tau +\omega.
$$

We first note that since $F_\tau$ is homotopic to the identity, $\omega_\tau$ is always cohomologous to $\omega$ in $H^{1,1}(X)$. By the $\ddbar$-lemma for K\"ahler manifolds this implies that we can always solve, for each $\tau$
\be
i\ddbar u_\tau =\omega_\tau -\omega,
\ee
and it is easy to see that we can make $u_\tau$ depend smoothly on $\tau$. The solutions $u_\tau$ are uniquely determined up to constants for each $\tau$, so as a function of $\tau$ we have uniqueness up to the addition of a function $f(\tau)$. The question is if this function can be chosen so that $u_\tau$ is a geodesic, i. e. that it is $\omega$-plurisubharmonic on $\C\times X$ and satisfies
$$
(i\ddbar_{\tau x} u_\tau +\omega)^{n+1}=0.
$$
Let $u_\tau$ be an arbitrary solution to (3.1), and let 
$$
\Omega:= i\ddbar_{\tau x} u_\tau +\omega.
$$
Let $\alpha_\tau:= \dbar_x\dot u_\tau$, where $\dot u_\tau=\partial u_\tau/\partial \tau$. This is a $(0,1)$-form on $X$ for each $\tau$ and it depends smoothly on $\tau$. A direct computation shows that
\be
\Omega =\omega_\tau +\ddot u_{\tau \bar\tau} id\tau\wedge d\bar\tau + i(\bar\alpha_\tau\wedge d\bar\tau +d\tau\wedge \alpha_\tau),
\ee
and
$$
\Omega^{n+1}/(n+1)!=(\ddot u_{\tau \bar\tau} -|\dbar_x \dot u_\tau|^2_\tau) id\tau\wedge d\bar\tau\wedge\omega_\tau^n/n!.
$$
Here $\ddot u_{\tau \bar\tau} =\partial^2 u_\tau/\partial \tau\partial\bar\tau$. We will use the simplifying notation
$$
\ddot u_{\tau \bar\tau} -|\dbar_x \dot u_\tau|^2_\tau=:c(u_\tau).
$$
From this formula we see that if we change $u_\tau$ by adding a function $f(\tau)$, we change $c(u)$ by adding $f''(\tau)$. Hence we are able to choose $u_\tau$ so that it satisfies the HCMAE if and only if $c(u_\tau)$ depends only on $\tau$, or equivalently if $\dbar_xc(u_\tau)=0$. We will prove
\begin{thm} Assume the holomorphic vector field $V$ satisfies
\be
V\rfloor\omega=i\dbar h
\ee
for some function $h$. Then there is a geodesic $u_\tau$ such that
$$
F_\tau^*(\omega)=i\ddbar u_\tau +\omega.
$$
\end{thm}
\begin{proof}
Since $V$ is invariant under the flow $F_\tau$, it follows from (3.3) that 
\be
V\rfloor\omega_\tau=i\dbar h\circ F_\tau.
\ee
On the other hand, by Cartan's 'magic formula' for the Lie derivative
$$
\partial V\rfloor\omega_\tau= \frac{\partial}{\partial \tau}\omega_\tau=i\ddbar\dot u_\tau.
$$
Hence
$$
\ddbar_x(h\circ F_\tau -\dot u_\tau)=0,
$$
so $h\circ F_\tau -\dot u_\tau$ is constant on $X$ for each $\tau$. By (3.4) this gives
\be
V\rfloor\omega_\tau=i\dbar \dot u_\tau.
\ee
Next we apply $\partial/\partial \bar \tau$ to this and get
$$
V\rfloor i\ddbar\dot u_{\bar\tau}=i\dbar\ddot u_{\tau \bar\tau}.
$$
This gives that 
\be
\dbar(V\rfloor\partial\dot u_{\bar\tau} -\ddot u_{\tau \bar\tau})=0.
\ee
But taking the conjugate of (3.5) we get that 
$$
V\rfloor\partial\dot u_{\bar\tau}=V\rfloor\bar V\rfloor\omega_\tau=|V|_\tau^2=|\dbar\dot u_\tau|^2_\tau.
$$
Hence (3.5) says precisely that $\dbar c(u_\tau)=0$, so we are done.
\end{proof}

{\bf Remark:} If condition (3.3) is satisfied for a certain $\omega$, it also holds for any other K\"ahler form $\omega'=\omega+i\ddbar v$ in the same cohomology class, with $h$ replaced by $h+V(v)$. In particular it also holds for $G^*(\omega)$ where $G$ is the time-1 flow of any holomorphic vector field. \qed

Let us note  that even though the geodesic is not unique, it can be chosen in a canonical way. It is clear that two geodesics differ only by a function of $\tau$ and this function has to be linear. Now consider the Aubin-Yau energy of $u_\tau$,
$$
\E(u_\tau)
$$
defined by $\E(u_0)=0$ and 
$$
\partial\E(u_\tau)/\partial \tau=\int_X \dot u_\tau \omega_\tau^n/n!.
$$
It is classical that this function is linear along geodesics (this follows also from the one-dimensional case of Theorem 2.1, since this implies that the derivative is constant), so by subtracting a suitable linear function of $\tau$ from the geodesic we can make it vanish. This is the 'canonical' choice, and it has the merit of removing all ambiguity from the definition. Theorem 3.1 can be restated as saying that if $V$ satisfies condition (3.3), and for each $\tau$, $u_\tau$ solves equation (3.1) and has energy zero, then $u_\tau$ is a geodesic.

We also remark that the cohomological condition (3.3) is not necessary for a geodesic to exist. A simple example is when $X$ is the standard torus; the quotient of $\C$ by the integer lattice $\Z^2$. Take $\omega$ to be $idz\wedge d\bar z$ on the torus, and $V=\partial/\partial z$. Then $V\rfloor\omega =i d\bar z$, which is not $\dbar$-exact on the torus. But,  $\omega_\tau=\omega$ for all $\tau$, so $u_\tau=0$ is a geodesic satisfying our conditions. 

On the other hand, a geodesic as in Theorem 3.1 does not always exist. For this we take the same torus as above and the same vector field, but choose instead
$$
\omega= \xi(x)idz\wedge d\bar z.
$$
Then $\omega_\tau=\omega_t$ is independent of the imaginary part of $\tau$ and periodic in $t$, but not constant, if $\xi$ is not constant, so there can be no geodesic. Indeed, the 'canonical' geodesic would then also be periodic in $t$, which by convexity would imply that it is constant in $t$. This example was communicated to me by Tam\'as Darvas.

Assume next that the imaginary part of $V$ is Hamiltonian for the symplectic form $\omega$, i. e. that
$$
(V- \bar V)\rfloor \omega = id H,
$$
for a real valued function $H$; the Hamiltonian. Identifying terms of the same bidegree, we see that $V\rfloor \omega=i\dbar H$, so $V$ satisfies the condition in Theorem 3.1 and therefore induces a geodesic, $u_\tau$. In the course of the proof of Theorem 3.1 we saw also that
$$
V\rfloor\omega=i\dbar \dot u_0.
$$
Hence  $\dot u_0$ is a Hamiltonian for the imaginary part of $V$ and the symplectic form $\omega$, and by (3.5), $\dot u_t$ is a Hamiltonian  for $\Im(V)$ for the symplectic form $\omega_t$.

We next turn to multidimensional geodesics. If $V$ is any holomorphic vector field on $X$, we denote by $F^V\in Aut(X)$ its time-1 flow. It is clear that if $V$ and $W$ are commuting holomorphic fields, then 
$$
F^{V+W}=F^V\circ F^W,
$$
since $F^{t(V+W)}$ and $F^{tV}\circ F^{tW}$ satisfy the same ODE. Let $V_1, ...V_k$ be commuting holomorphic vector fields satisfying the cohomological condition of Theorem 3.1. Put
$$
F^{(\tau_1, ...\tau_k)}:= F^{\tau_1V_1+...\tau_k V_k},
$$
$$
\omega_{(\tau_1, ...\tau_k)}:= (F^{(\tau_1, ...\tau_k)})^*(\omega).
$$
and solve
$$
i\ddbar u_{(\tau_1, ...\tau_k)}=\omega_{(\tau_1, ...\tau_k)}-\omega
$$
with $u_{(\tau_1, ...\tau_k)}$ having energy zero. We claim that the restriction of $u_{(\tau_1, ...\tau_k)}$
to any complex line is a complex geodesic. Indeed, this is clear since if the line is $\{a+\tau b\}$ then 
$$
\omega_{a+\tau b}= (F^{\tau b})^* (F^{a})^*(\omega)
$$
satisfies the hypothesis of Theorem 3.1 (by the remark immediately after its proof). In particular, this means that if $a$ and $b$ lie in $\R^k$, then $u_{a+tb}$ is a real geodesic if the imaginary parts of $V_j$ are all Hamiltonian. In other words, $u_{t_1, ...t_k}$ is a multidimensional geodesic.

Moreover its gradient satisfies
$$
\partial_t u_t|_0=: (H_1, ... H_k),
$$
where $H_i$ is a Hamiltonian for $V_i$. In other words, the gradient of $u_t$ at $t=0$ is a moment map for $V$. Thus Theorem 2.2 has the following corollary, which is a special case of the results of Atiyah, \cite{Atiyah}, and Guillemin-Sternberg, \cite{Guillemin-Sternberg}.
\begin{cor} Let $V=(V_1, ...V_k)$ be a $k$-tuple of commuting holomorphic vector fields, whose imaginary parts are Hamiltonian for the symplectic (K\"ahler) form $\omega$. Then the image of its moment map is a convex subset of $\R^k$.
\end{cor}

\section{ Geodesics that define holomorphic vector fields}
The main objective of this section is to prove
\begin{thm} Let $u_\tau$ be a regular geodesic in $\H_\omega$ defined for all $\tau\in \C$. Assume moreover that $u_\tau$ does not depend on $\Im(\tau)$.  Then there is a holomorphic vector field on $X$ that induces $u_\tau$ in the sense of the previous section.
\end{thm}

In the proof of the theorem we will use the Mabuchi K-energy, $K: \H_\omega \to \R$, see \cite{2Mabuchi}. It is a classical fact that if $u_\tau$ is a regular complex geodesic, then $K(u_\tau)$ is a subharmonic function of $\tau$ (this was later shown to also hold for generalized geodesics with enough regularity in \cite{BB} and \cite{CLP}). Moreover, if $K$ is harmonic along a regular geodesic, then the geodesic is induced by a holomorphic vector field, \cite{Donaldson}.  This vector field may be time dependent, i. e. its coefficients may depend on $\tau$, but it is a holomorphic function of $\tau$, see e.g. \cite{B2}. If the geodesic is independent of the imaginary part of $\tau$, the same thing goes for the vector field, which must therefore be independent of $\tau$, and induce the geodesic in the sense of the previous section. 

\begin{thm} Let $u_\tau$ be a regular geodesic defined for $\tau$ in a domain $U$ in $\C$.  Let
$$
\Theta:= \frac{\partial^2}{\partial\tau\partial\bar\tau} K(u_\tau) id\tau\wedge d\bar\tau,
$$
where $K$ is the Mabuchi K-energy. Then either $\Theta$ is identically equal to zero, or $\Theta$ defines a (possibly singular) metric on $U$ with negative curvature, bounded from above by $-(2/n V(X))$ where  $V(X)$ is the volume of $X$ for the K\"ahler metric $\omega$ .
\end{thm}
To explain the statement of the theorem, recall that if $e^v id\tau\wedge d\bar\tau$ is the K\"ahler form of a metric on a domain in $\C$, then its curvature form is
$$
-i\ddbar v.
$$
We say that a possibly singular metric has curvature less than $-C$ if its curvature form is bounded from above by $-C$ times the metric form, or in other words if $v$ is subharmonic and satisfies
$$
i\ddbar v\geq C e^v id\tau\wedge d\bar\tau.
$$
Since there are no metrics on $\C$ of strictly negative curvature, it is clear that Theorem 4.1 follows from Theorem 4.2. 
For a general domain $U$, we can also compare $\Theta$ to Ahlfors' ultrahyperbolic metric, \cite{Ahlfors} and get the following corollary.
\begin{cor} For a regular geodesic, defined in a domain in $\C$, the Laplacian of the K-energy is bounded from above by $n/2V(X)$ times the ultrahyperbolic metric of the domain. In particular, for a regular geodesic ray, $u_t$ (defined for $t>0$) 
$$
\frac{d^2}{dt^2}K(u_t) \leq (2n/V(X))t^{-2}.
$$
\end{cor}
Thus we get uniform bounds of the second derivative of the K-energy, depending only on the dimension and the volume of $X$. It seems it would be interesting to know if this holds also for generalized geodesics.

The proof of Theorem 4.2 depends heavily on a remarkable result of Burns, \cite{Burns}. A convenient starting point for our discussion is the formula for the $i\ddbar$ of the K-energy from \cite{BB}
\be
i\ddbar K(u_\tau)= p_*( (i\ddbar\log (\Omega^n\wedge id\tau\wedge d\bar\tau)\wedge \Omega^n/n!).
\ee
This needs to be explained a bit: First, as above,
$$
\Omega=i\ddbar_{\tau x} u_\tau +\omega.
$$
Second
$$
\Omega^n\wedge id\tau\wedge d\bar\tau
$$
is a form of top degree on $\C\times X$ and when we take its logarithm we mean that we take  the logarithm of its density with respect to $  c_n v\wedge\bar v$ where $v$ is a nonvanishing holomorphic form of maximal degree. This depends on the choice of $v$, so the density is not well defined, but the $\ddbar$ of the logarithm does not depend on the choice of $v$.
Finally, $p$ is the projection map from $\C\times X$ to $\C$. In the sequel we will put
$$
\theta:=i\ddbar\log (\Omega^n\wedge id\tau\wedge d\bar\tau).
$$

 Since $u_\tau$ is a regular geodesic it is smooth and strictly $\omega$-plurisubharmonic on $X$ for each $\tau$. This means that $\Omega$ is strictly positive on $\{\tau\}\times X$. Since $\Omega^{n+1}=0$,   $\Omega$  defines a foliation by holomorphic graphs of functions
$$
f^x(\tau):\C\to X, \quad f^x(0)=x,
$$
for each $x$ in $X$,
along which $\Omega$ vanishes.  Denote by $Y_x=\{(\tau, f^x(\tau)\}$ the graph of $f^x$.
\begin{prop}Let $[Y_x]$ be the current of integration on $Y_x$. Then 
$$
\Omega^n/n!= \int_X [Y_x] \omega^n/n!.
$$
\end{prop}
In the proof of the proposition we will use a well known lemma.
\begin{lma} Let $f_\tau (x):= f^x(\tau)$. Then
$$
f_\tau^*(\omega_\tau)=\omega_0=\omega.
$$
\end{lma}
\begin{proof}
Define first a time dependent vector field $V_\tau$ on $X$ by
$$
V_\tau\rfloor\omega_\tau=i\dbar \dot u_\tau,
$$
where, as before $\dot u_\tau=\partial u_\tau/\partial \tau$. Then, let
$$
\V:=\partial/\partial\tau -V_\tau
$$
(a vector field on $\C\times X$). It follows from formula (3.2) that $\V\rfloor\Omega=0$. Hence  $\V$ is tangent to the foliation defined by $\Omega$, so
$$
\frac{\partial f_\tau}{\partial\tau}= -V_\tau.
$$
This gives
$$
\frac{\partial f_\tau^*(\omega_\tau)}{\partial\tau}=f_\tau^*\left(\L_{-V_\tau}(\omega_\tau)+\dot\omega_\tau\right)=
f_\tau^*\left( -di\dbar \dot u_\tau -i\ddbar\dot u_\tau\right)=0.
$$
Therefore $f_\tau^*(\omega_\tau)$ is constant; hence equal to $\omega$ since $f_0$ is the identity map.
\end{proof}
We can now continue with the proof of the proposition. The statement means that if $\Theta$ is a form of bidegree $(1,1)$, then
\be
\int\Theta\wedge \Omega^n/n!=\int_X \left(\int_{Y_x} \Theta\right) \omega^n/n!.
\ee
Parametrizing the leaf $Y_x$ by the map $\tau\to (\tau,f_\tau(x))=:F_\tau(x)$ we have
$$
\int_{Y_x} \Theta=\int_U \Theta(\V,\bar\V)(F_\tau(x))id\tau\wedge d\bar\tau.
$$
On the other hand we have the pointwise formula
$$
\Theta\wedge\Omega^n/n!=\Theta(\V,\bar\V)\Omega^n/n!\wedge id\tau\wedge d\bar\tau=\Theta(\V,\bar\V)\omega_\tau^n/n!\wedge id\tau\wedge d\bar\tau,
$$
which can be seen by contracting the identity
$$
\Theta\wedge \Omega^n/n! id\tau\wedge d\bar\tau=0
$$
by first $\bar\V$ and then $\V$, and using that $\V\rfloor\Omega=0$ and $\V\rfloor d\tau=1$. 

Hence we get in view of the lemma that
$$
\int_X \left(\int_{Y_x} \Theta\right) \omega^n/n!=\int_X\left(\int F_\tau^*(\Theta(\V,\bar \V)\omega_\tau^n/n!)\right) id\tau\wedge d\bar\tau=
$$
$$
=\int d\tau\wedge d\bar\tau\int\Theta(\V,\bar\V)\omega_\tau^n/n!=\int\Theta\wedge \Omega^n/n!.
$$
This completes the proof of the proposition.

By the proposition and formula (4.1) we have 
\be
i\ddbar K(u_\tau)= \int_X p_*(\theta\wedge [Y_x])\omega^n/n!.
\ee
But 
\be
p_*(\theta\wedge [Y_x]) = G_x^*(\theta),
\ee
where $G_x:\C \to \C\times X$ is the parametrization of the leaf $Y_x$, so (4.3) expresses $i\ddbar K(u_\tau)$ as a superposition of the forms $\theta_x:=G_x^*(\theta)$ with respect to the measure $\omega^n/n!$ on $X$. 

This is where Burns' result enters the picture. In an earlier work by Bedford and Burns, \cite{Bedford-Burns}, it was proved that $\theta$ is a positive form on each leaf. This means that the K\"ahler volume forms $\omega_\tau^n$ are log-subharmonic along each leaf. They also found an explicit formula for $i\ddbar\log \omega_\tau^n$. Then, in \cite{Burns}, Burns proved that if we interpret this form as the K\"ahler form of a metric on $Y_x$, the Gaussian curvature of this metric is negative and bounded from above by $-2/n$. Thus the metric defined by $i\ddbar K(u_\tau)$ is the superposition of metrics of metrics with curvature less that $-2/n$ with respect to a positive measure of total mass equal to the volume of $X$ for the metric defined by  $\omega$. Theorem (4.2) therefore follows from the following 
\begin{prop} Let $g_\alpha$ be (possibly singular) metrics on a domain in $\C$ with curvature bounded from above by a negative constant $-a$. Let
$$
g :=\int_\alpha g_\alpha  d\nu(\alpha),
$$
where $d\nu$ is a positive measure of total mass $C$. Then the curvature of $g$ is bounded from above by $-a/C$.
\end{prop}
\begin{proof}

Recall that the hypothesis on the curvature of $g_\alpha$ means that
$$
\Delta \log g_\alpha \geq a g_\alpha,
$$
where $\Delta=\partial^2/\partial\tau\partial\bar\tau$. Equivalently 
\be
\Delta g_\alpha \geq a g_\alpha^2 +\frac{|\partial g_\alpha|^2}{g_\alpha}.
\ee
By Cauchy's inequality
$$
|\partial g|^2\leq\int\frac{|\partial  g_\alpha|^2}{g_\alpha}d\nu(\alpha)\int g_\alpha d\nu(\alpha)=g\int\frac{|\partial  g_\alpha|^2}{g_\alpha}d\nu(\alpha),
$$
and
$$
g^2\leq C\int g_\alpha^2 d\nu(\alpha).
$$
Thus the claim follows by integrating (4.5).

\end{proof}

\end{document}